\documentclass[11pt,a4paper]{article}
\usepackage[left=2.0cm, top=2.45cm,bottom=2.45cm,right=2.0cm]{geometry}
\usepackage{mathtools,amssymb,amsthm,mathrsfs,calc,graphicx,xcolor,cleveref,dsfont,tikz,pgfplots,bm}
\usepackage[british]{babel}
\usepackage{amsfonts}              % for blackboard bold, etc
\usepackage[T1]{fontenc}
\numberwithin{equation}{section}
\newtheorem {theorem}{Theorem}[section]
\newtheorem {proposition}[theorem]{Proposition}

\newtheorem {corollary}[theorem]{Corollary}

{\theoremstyle{definition}

}
{\theoremstyle{theorem}
\newtheorem {remark}[theorem]{Remark}

}

\newcommand{\Ent}{\operatorname{Ent}}
\newcommand{\dint}{\textup{d}}
\newcommand{\SO}{\textup{SO}}
\newcommand{\conv}{\textup{conv}}
\newcommand{\supp}{\textup{supp}}

\def\EE{\mathbb{E}}

\def\MM{\mathbb{M}}
\def\NN{\mathbb{N}}

\def\PP{\mathbb{P}}
\def\QQ{\mathbb{Q}}
\def\RR{\mathbb{R}}
\def\SS{\mathbb{S}}

\def\VV{\mathbb{V}}

\def\XX{\mathbb{X}}

\def\sfN{{\sf N}}

\def\cC{\mathcal{C}}

\def\cF{\mathcal{F}}

\def\cH{\mathcal{H}}

\def\cN{\mathcal{N}}

\def\cX{\mathcal{X}}

\def\sC{\mathscr{C}}

\def\sH{\mathscr{H}}

\begin{document}

\title{\bfseries Concentration on Poisson spaces\\ via modified $\Phi$-Sobolev inequalities}

\author{Anna Gusakova\footnotemark[1],\; Holger Sambale\footnotemark[2]\;\; and Christoph Th\"ale\footnotemark[3]}

%\rhead{Concentration inequalities for Poisson cylinder processes}

\date{}
\renewcommand{\thefootnote}{\fnsymbol{footnote}}
\footnotetext[1]{Ruhr University Bochum, Germany. Email: anna.gusakova@rub.de}

\footnotetext[2]{Bielefeld University, Germany. Email: hsambale@math.uni-bielefeld.de}

\footnotetext[3]{Ruhr University Bochum, Germany. Email: christoph.thaele@rub.de}

\maketitle

\begin{abstract} Concentration properties of functionals of general Poisson processes are studied. Using a mo\-di\-fied $\Phi$-Sobolev inequality a recursion scheme for moments is established, which is of independent interest. This is applied to derive moment and concentration inequalities for functionals on abstract Poisson spaces. Applications of the general results in stochastic geometry, namely Poisson cylinder models and Poisson random polytopes, are presented as well. \\[2mm]
\end{abstract}

%\tableofcontents

\section{Introduction}

The stochastic analysis of non-linear functionals of Poisson processes (so-called Poisson functionals) on general state spaces was subject of intensive investigation during the last decade. New tools have been developed and new applications have been found. Most notably, we mention here the very fruitful connection between Malliavin calculus on Poisson spaces with Stein's method for normal approximation and the various striking applications in stochastic geometry of the resulting abstract limit theorems. For further background material on this topic and for references we refer to the monograph of Last and Penrose \cite{LP} as well as to the collection edited by Peccati and Reitzner \cite{PeccatiReitzner}. In contrast to the weak limit theorems just mentioned we are interested in concentration properties of Poisson functionals. In the past questions of this type for general Poisson functionals have been approached from three different angles. First, Wu \cite{Wu} developed a modified log-Sobolev inequality to derive a collection of concentration bounds under rather restrictive conditions. This approach was largely extended by Bachmann and Peccati \cite{BP}. In that paper and in \cite{BR} various applications of such bounds to random geometric graphs are discussed. Using general covariance identities for exponential functions of Poisson processes Gieringer and Last \cite{GieringerLast} were able to derive a new set of concentration inequalities. They were particularly useful to study concentration properties of geometric functionals associated with the Poisson Boolean model \cite{GieringerLast} or with so-called Poisson cylinder processes \cite{BaciBetkenGusakovaThaele} considered in stochastic geometry. In \cite{ReitznerConcentration} Reitzner proved an analogue of Talagrand's inequality for the convex distance on the Poisson space, which has found applications in \cite{ReitznerSchulteThaele} in the context of random geometric graphs. Finally, based on a general transportation inequality concentration bounds for so-called convex functionals of Poisson processes were proved by Gozlan, Herry and Peccati \cite{GHP}. They also gave applications to Poisson U-statistics.

It is the purpose of this paper to add another approach to concentration properties of Poisson functionals (see the beginning of Section \ref{sec:PhiSobolev} for a formal definition of this notion). It is based on a modified $\Phi$-Sobolev inequality for Poisson processes, which is due to Chafa\"{i} \cite{Chafai2004}. We formulate the result in Section \ref{sec:PhiSobolev} and demonstrate that these modified $\Phi$-Sobolev inequalities interpolate between two well-known results, namely the Poincar\'e inequality and Wu's modified log-Sobolev inequality on Poisson spaces. From the modified $\Phi$-Sobolev inequality we derive a recursive scheme of $L^p$-estimates for Poisson functionals by specializing the function $\Phi$. They involve two different types of difference operators, which reflect the discrete nature of the underlying problem, and are of independent interest. In Section \ref{sec:MomentConcentrationBounds} we use the recursive scheme to derive moment and concentration bounds for Poisson functionals. In particular, we recover qualitatively the bounds of Bachmann and Peccati by our method. The approach just described is not new and has been used by Boucheron, Bousquet, Lugosi and Massart \cite{BBLM} for functionals of independent random variables. This note demonstrates that a similar methodology can successfully be implemented in the framework of Poisson functionals as well. In the final Section \ref{sec:Appl} we discuss two novel applications of our general results to models from stochastic geometry, namely Poisson polytopes and Poisson cylinder processes. While concentration bounds for geometric functionals of Poisson cylinder models are known from \cite{BaciBetkenGusakovaThaele} and concentration inequalities for random polytopes can be found in \cite{Vu}, some of the estimates we prove are new. This is especially the case for the so-called intrinsic volumes of Poisson polytopes with vertices chosen from the boundary of a smooth convex body for which no concentration inequalities can be found in the existing literature.

We would like to emphasize that the general concentration inequalities we prove Section \ref{sec:MomentConcentrationBounds} of this paper are not new. Against this light, the main emphasis of our text is on providing an effective alternative approach based on moment estimates. In particular, the applications to the stochastic geometry models we develop in Section \ref{sec:Appl} could in principle also be derived from the general inequalities stated in \cite{BP}.

\medskip

Independently and, in a sense, in parallel with us Adamczak, Polaczyk and Strzelecki \cite{AdamczakEtAl} have very recently obtained (among many other results) moment estimates and concentration bounds for Poisson functionals, which are similar to the ones we prove, see especially Section 4.7 in \cite{AdamczakEtAl}. The main difference to their paper is that they use a so-called Beckner-type inequality as their starting point (a device we will be able to recover from our results as well, see Corollary \ref{cor:Beckner}), while we are building on a modified $\Phi$-Sobolev inequality. Apart from this, both approaches eventually rely on the methods of Boucheron, Bousquet, Lugosi and Massart \cite{BBLM} and lead to very similar results. We will further comment on the differences within the text.

\section{Modified $\Phi$-Sobolev inequalities for Poisson functionals}\label{sec:PhiSobolev}

Let $(\XX,\cX)$ be a measurable space supplied with a $\sigma$-finite measure $\mu$. By $\sfN(\XX)$ we denote the space of $\sigma$-finite counting measures on $\XX$. The $\sigma$-field $\cN(\XX)$ is defined as the smallest $\sigma$-field on $\sfN(\XX)$ such that the evaluation mappings $\xi\mapsto\xi(B)$, $B\in\cX$, $\xi\in\sfN(\XX)$ are measurable. A \textbf{point process} on $\XX$ is a measurable mapping with values in $\sfN(\XX)$ defined over some fixed probability space $(\Omega,\cF,\PP)$. By a \textbf{Poisson process} $\eta$ on $\XX$ with intensity measure $\mu$ we understand a point process with the following two properties:
\begin{itemize}
    \item[(i)] for any $B\in\cX$ the random variable $\eta(B)$ is Poisson distributed with mean $\mu(B)$;
    \item[(ii)] for any $n\in\NN$ and pairwise disjoint sets $B_1,\ldots,B_n\in\cX$ the random variables $\eta(B_1),\ldots,\eta(B_n)$ are independent.
\end{itemize}
A \textbf{Poisson functional} is a random variable $F$ $\PP$-almost surely satisfying $F=f(\eta)$ for some measurable $f:\sfN(\XX)\to\RR$. In this case $f$ is called a \textbf{representative} of $F$. If $\PP_\eta$ denotes the distribution of the Poisson process $\eta$ we will write $L^p(\PP_\eta)$, $p\geq 0$, for the space of Poisson functionals $F$ satisfying $\EE|F|^p<\infty$. For a Poisson functional $F$ with representative $f$ and $x\in\XX$ we define the \textbf{difference operator} $D_xF$ by putting
$$
D_xF:=f(\eta+\delta_x)-f(\eta),
$$
where $\delta_x$ stands for the Dirac measure at $x$. For further background on Poisson processes on general state spaces we refer the reader to the monograph \cite{LP}.

\medskip

We denote by $\sC$ the space of functions $\Phi:\RR_+\to\RR$ satisfying the following three properties:
\begin{itemize}
    \item[(i)] $\Phi$ is convex and continuous,
    \item[(ii)] $\Phi$ is twice differentiable on $(0,\infty)$,
    \item[(iii)] $\Phi$ is either affine, or $\Phi''$ is strictly positive and $1/\Phi''$ is concave.
\end{itemize}
Typical examples of functions belonging to $\sC$ are $\Phi_{\log}(x)=x\log x$ or $\Phi_r(x)=x^{2/r}$ with $r\in(1,2)$ and $x\in\RR_+$. For $\Phi\in\sC$ the \textbf{$\Phi$-entropy} of a random variable $F$ is defined as
$$
\Ent_\Phi(F) := \EE[\Phi(F)] - \Phi(\EE[F]).
$$
In particular, the classical entropy $\Ent(F)=\EE[F\log F]-\EE[F]\log(\EE[F])$ of $F$ is recovered by taking $\Phi=\Phi_{\log}$. 

The following modified $\Phi$-Sobolev inequality is due to Chafa\"{i} \cite[Section 5.1]{Chafai2004} and generalizes the modified log-Sobolev inequality of Wu \cite{Wu}. %We include the argument for completeness.

\begin{proposition}[Modified $\Phi$-Sobolev inequality]\label{prop:PhiSobolevIneq}
Let $\eta$ be a Poisson process on a measurable space $\XX$ with $\sigma$-finite intensity measure $\mu$. Let $F\in L^1(\PP_\eta)$, $\Phi\in\sC$ and assume that $F>0$ $\PP$-almost surely. Then
$$
\Ent_\Phi(F) \leq \EE\Big[\int_\XX \big(D_x\Phi(F)-\Phi'(F)D_xF\big)\,\mu(\dint x)\Big].
$$
\end{proposition}

\begin{remark}\rm 
If the derivative $\Phi'(0)$ of $\Phi$ at $0$ is well defined, we can relax the positivity assumption on $F$ in the previous proposition by requiring only that $F\geq 0$ $\PP$-almost surely. We will take advantage of this fact from now on whenever we are dealing with $\Phi_r$.
\end{remark}

\begin{corollary}\label{cor:PhiRSovolev}
Fix $r\in(1,2)$ and let $F\in L^1(\PP_\eta)$ be a Poisson functional satisfying $F\geq 0$ $\PP$-almost surely. Then
$$
\Ent_r(F):=\Ent_{\Phi_r}(F) \leq \EE\Big[\int_\XX\big(D_xF^{2\over r}-{2\over r}F^{{2\over r}-1}\,(D_xF)\big)\,\mu(\dint x)\Big].
$$
\end{corollary}
\begin{proof}
This is a direct consequence of Proposition \ref{prop:PhiSobolevIneq} with $\Phi=\Phi_r$, since $\Phi'(x)={2\over r}x^{{2\over r}-1}$.
\end{proof}

Note that if we let $r\to 1$, then $\Ent_r(F)\to\VV(F)$, the variance of $F$. Using  that for two Poisson functionals $F,G$ and for all $x\in\XX$ one has the product rule
$$
D_x(FG) = (D_xF)G+F(D_xG) + (D_xF)(D_xG)
$$
(compare with \cite[Exercise 18.2]{LP}), it follows that, for $r=1$,
\begin{align*}
    D_xF^{2\over r}-{2\over r}F^{{2\over r}-1}\,(D_xF) =  D_xF^2 -2 F\,D_xF = 2F(D_xF)+(D_xF)^2-2F(D_xF) = (D_xF)^2.
\end{align*}
Thus, as $r\to 1$, the modified $\Phi_r$-Sobolev inequality turns into the $L^1$-version of the \textbf{Poincar\'e inequality} for Poisson functionals:
\begin{equation}\label{eq:PoincareIneq}
\VV(F) \leq \EE\Big[\int_\XX(D_xF)^2\,\mu(\dint x)\Big],    
\end{equation}
see \cite[Proposition 2.5]{LastPeccatiSchulte} and \cite[Theorem 18.7]{LP}. On the other hand, as $r\to 2$, we have that, for $x\in\RR_+$,
\begin{align}
    x^{2\over r} &= x-\Big({r\over 2}-1\Big)x\log x + O((r-2)^2),\label{eq:xHoch2/r}\\
    \nonumber -{2\over r}x^{{2\over r}-1} &= -1 + \Big({r\over 2}-1\Big)(\log x+1) + O((r-2)^2),
\end{align}
which implies that $\PP$-almost surely
\begin{align*}
    D_xF^{2\over r}-{2\over r}F^{{2\over r}-1}\,(D_xF)
    &= D_x\Big(F-\Big({r\over 2}-1\Big)F\log F\Big) - D_xF + \Big({r\over 2}-1\Big)(D_xF)(\log F+1)+O((r-2)^2)\\
    &= \Big(1-{r\over 2}\Big)\Big(D_x(F\log F)-(\log F+1)(D_xF)+ O(r-2)\Big) .
\end{align*}
As a consequence, we conclude that on the one hand side, as $r\to 2$,
\begin{align*}
\Big(1-{r\over 2}\Big)^{-1}\EE\Big[\int_\XX\big(D_xF^{2\over r}-{2\over r}F^{{2\over r}-1}\,(D_xF)\big)\,\mu(\dint x)\Big]
\to \EE\Big[\int_\XX\big(D_x(F\log F)-(\log F+1)(D_xF)\big)\,\mu(\dint x)\Big].    
\end{align*}
On the other hand, using \eqref{eq:xHoch2/r} again we see that, as $r\to 2$,
\begin{align*}
    \Big(1-{r\over 2}\Big)^{-1}\Ent_r(F) = \Big(1-{r\over 2}\Big)^{-1}\Big(\EE[F^{2\over r}]-(\EE[F])^{2\over r}\Big)
    = \EE[F\log F]-\EE[F]\log(\EE[F]) + O(r-2),
\end{align*}
which implies 
$$
\Big(1-{r\over 2}\Big)^{-1}\Ent_r(F) \to \Ent(F),
$$
as $r\to 2$, where $\Ent(F)=\Ent_{\Phi_{\log}}$ is the classical entropy of $F$.
Summarizing, we conclude that, as $r\to 2$, the modified $\Phi_r$-Sobolev inequality turns into Wu's {\bf modified log-Sobolev inequality}
\begin{equation}\label{eq:LogSobolve}
\begin{split}
\Ent(F) &\leq \EE\Big[\int_\XX \big(D_x(F\log F)-(\log F+1)(D_xF)\big)\,\mu(\dint x)\Big]\\
&=\EE\Big[\int_\XX \big(D_x\Phi_{\log}(F)-\Phi_{\log}'(F)(D_xF)\big)\,\mu(\dint x)\Big],
\end{split}
\end{equation}
see \cite[Theorem 1.1]{Wu}. Against this light one can say that family of modified $\Phi_r$-Sobolev inequalities for Poisson functionals from Corollary \ref{cor:PhiRSovolev} interpolates between two classical inequalities for Poisson functionals, namely the Poincar\'e inequality ($r\to 1$, see \eqref{eq:PoincareIneq}) and the modified log-Sobolev inequality ($r\to 2$, see \eqref{eq:LogSobolve}) for Poisson functionals.

\medskip

Note that the modified $\Phi_r$-Sobolev inequality established in Corollary \ref{cor:PhiRSovolev} implies the {\bf Beckner-type inequality} Bec-$(2/r)$ used in \cite[Section 4.7]{AdamczakEtAl}.

\begin{corollary}[Beckner-type inequality]\label{cor:Beckner}
Let $F\in L^1(\PP_\eta)$ satisfy $F\geq 0$ $\PP$-almost surely and fix $r\in(1,2)$. Then $F$ satisfies the Beckner-type inequality Bec-$(2/r)$, meaning that
\[
\Ent_r(F)\leq {6\over r} \EE\Big[\int_\XX(D_xF)\,(D_xF^{{2\over r}-1})\,\mu(\dint x)\Big].
\]
\end{corollary}
\begin{proof}
This immediately follows from Corollary \ref{cor:PhiRSovolev} once we have checked that
\[
\EE\Big[\int_\XX\big(D_xF^{2\over r}-{2\over r}F^{{2\over r}-1}\,(D_xF)\big)\,\mu(\dint x)\Big] \le {6\over r} \EE\Big[\int_\XX(D_xF)\,(D_xF^{{2\over r}-1})\,\mu(\dint x)\Big].
\]
Setting $p := 2/r$, the latter follows from the pointwise inequality
\[
a^p - b^p - p b^{p-1}(a-b) \le 3p (a^p + b^p - ab^{p-1} - ba^{p-1})
\]
for any $a, b \ge 0$. Indeed, fixing $b \ge 0$ and rearranging, this inequality is equivalent to $\psi(a) \ge 0$ for $a \ge 0$, where
\[
\psi(a) := (3p-1)a^p + (2p+1)b^p - 2p ab^{p-1} - 3p ba^{p-1}.
\]
Now, for $a \ge 0$, one has that
\[
\psi'(a) = (3p-1)pa^{p-1} - 2p b^{p-1} - 3p(p-1) ba^{p-2},
\]
which has a unique zero at $a = b$ with $\psi(b) = 0$. From this fact the claim follows.
\end{proof}

Finally, let us stress that the various functional inequalities discussed in this section are closely related (but not completely equivalent) concepts. For instance, it follows from \cite[Theorem 1.1]{AdamczakEtAl} that the family of Beckner-type inequalities from the previous corollary is equivalent to the modified logarithmic Sobolev inequality
\[
\Ent(F) \le \int_\XX \EE[(D_xF)(D_x\log F)]\,\mu(\dint x).
\]
As demonstrated in Section 4.7 in \cite{AdamczakEtAl}, the latter is an easy consequence of Wu's modified log-Sobolev inequality.

\section{$L^p$-estimates for Poisson functionals}\label{sec:LpEstimates}

This section is devoted to $L^p$-estimates for functionals $F$ of Poisson processes on a measurable space $\XX$. For this we use the modified $\Phi$-Sobolev inequality for Poisson functionals developed in the previous section. As above we assume that $\eta$ is a Poisson process on $\XX$ and denote by $\mu$ its $\sigma$-finite intensity measure. For $x\in\RR$ we shall write $x_+:=\max\{0,x\}$ and $x_-:=\min\{0,x\}$ and for $p>0$ and a random variable $F$ we put $\|F\|_p:=(\EE[|F|^p])^{1/p}$.

\begin{proposition}[Recursive $L^p$-estimate]\label{prop:LpEstimate}
Let $p\geq 2$ and $F\in L^{1}(\PP_\eta)$ be a Poisson functional. Then
\begin{equation}\label{Lpest}
\lVert (F - \EE[F])_+ \rVert_p^p \le \lVert (F - \EE[F])_+\rVert_{p-1}^p + (p-1) \lVert V^+ \rVert_{p/2} \lVert (F - \EE[F])_+ \rVert_p^{p-2},
\end{equation}
where
\[
V^+ := \int_\XX (D_x F)_-^2 \,\mu (\dint x) + \int_\XX (F(\eta) - F(\eta - \delta_x))_+^2 \,\eta(\dint x).
\]
\end{proposition}

\begin{proof}
If $(F-\EE[F])_+\notin L^{p-1}(\PP_\eta)$ there is nothing to prove. Hence, from now on we assume that $(F-\EE[F])_+\in L^{p-1}(\PP_\eta)$. So, given the functional $F$ with representative $f$ consider another Poisson functional $G$ with representative $g$ defined as 
\[
g(\eta):=(f(\eta)-\EE[F])_+^{p-1}.
\]
It is clear that $G\geq 0$ $\PP$-almost surely. Thus, applying Corollary \ref{cor:PhiRSovolev} to the Poisson functional $G$ and with $r=2-{2\over p}$ we obtain 
\begin{align}
\Ent_{2-2/p}(G)&=\EE[(F-\EE[F])_+^{p}]-(\EE[(F-\EE[F])_+^{p-1}])^{p\over p-1}\notag\\
&=\|(F-\EE[F])_+\|^p_p-\|(F-\EE[F])_+\|^{p}_{p-1}\notag\\
&\leq \EE\Big[\int_\XX \big(D_x(F-\EE[F])_+^{p}-{p\over p-1}(F-\EE[F])_+D_x(F-\EE[F])_+^{p-1}\big)\,\mu(\dint x)\Big].\label{eq:Lp1}
\end{align}
Using the definition of difference operator $D_x$ we get
\begin{align}
I(\eta,x):&=D_x(F-\EE[F])_+^{p}-{p\over p-1}(F-\EE[F])_+D_x(F-\EE[F])_+^{p-1}\notag\\
&=(f(\eta+\delta_x)-\EE[F])_+^{p}-(f(\eta)-\EE[F])_+^{p}-{p\over p-1}(f(\eta)-\EE[F])_+(f(\eta+\delta_x)-\EE[F])_+^{p-1}\notag\\
&\qquad +{p\over p-1}(f(\eta)-\EE[F])_+^{p}\notag\\
&={1\over p-1}(f(\eta)-\EE[F])_+\Big[(f(\eta)-\EE[F])_+^{p-1}-(f(\eta+\delta_x)-\EE[F])_+^{p-1}\Big]\notag\\
&\qquad+(f(\eta+\delta_x)-\EE[F])_+^{p-1}\Big[(f(\eta+\delta_x)-\EE[F])_+-(f(\eta)-\EE[F])_+\Big].\label{eq:Lp2}
\end{align}
By the mean value theorem for the function $x^{p-1}$ on the interval $J:=\Big[\min((f(\eta)-\EE[F])_+, (f(\eta+\delta_x)-\EE[F])_+), \max((f(\eta)-\EE[F])_+, (f(\eta+\delta_x)-\EE[F])_+)\Big]$ we have
%\begin{equation}\label{eq:Lp3}
%x^{n}-y^{n}=(x-y)\sum_{i=0}^{n-1}x^{i}y^{n-1-i},
%\end{equation}
%which is valid for all $x,y\geq 0$ and $n\geq 1$, we get
%\begin{align*}
%{1\over p-1}&\Big[(f(\eta)-\EE[F])_+^{p-1}-(f(\eta+\delta_x)-\EE[F])_+^{p-1}\Big]\\
%&=-{1\over p-1}\Big[(f(\eta+\delta_x)-\EE[F])_+-(f(\eta)-\EE[F])_+\Big]\sum_{i=0}^{p-2}(f(\eta+\delta_x)-\EE[F])_+^{i}(f(\eta)-\EE[F])_+^{p-2-i}.
%\end{align*}
\begin{align*}
{1\over p-1}\Big[(f(\eta)-\EE[F])_+^{p-1}-(f(\eta+\delta_x)-\EE[F])_+^{p-1}\Big]
=-\Big[(f(\eta+\delta_x)-\EE[F])_+-(f(\eta)-\EE[F])_+\Big]c^{p-2}
\end{align*}
for some $c\in J$. Now consider two cases. If $(f(\eta+\delta_x)-\EE[F])_+-(f(\eta)-\EE[F])_+\ge 0$, then $(f(\eta+\delta_x)-\EE[F])_+\ge (f(\eta)-\EE[F])_+$ and we obtain
\begin{align*}
&{1\over p-1}\Big[(f(\eta)-\EE[F])_+^{p-1}-(f(\eta+\delta_x)-\EE[F])_+^{p-1}\Big]\\
&\qquad\leq-\Big[(f(\eta+\delta_x)-\EE[F])_+-(f(\eta)-\EE[F])_+\Big](f(\eta)-\EE[F])_+^{p-2}.
\end{align*}
On the other hand, if $(f(\eta+\delta_x)-\EE[F])_+-(f(\eta)-\EE[F])_+\leq 0$, then $(f(\eta+\delta_x)-\EE[F])_+\leq (f(\eta)-\EE[F])_+$ and we again obtain
\begin{align*}
&{1\over p-1}\Big[(f(\eta)-\EE[F])_+^{p-1}-(f(\eta+\delta_x)-\EE[F])_+^{p-1}\Big]\\
&\qquad\leq-\Big[(f(\eta+\delta_x)-\EE[F])_+-(f(\eta)-\EE[F])_+\Big](f(\eta)-\EE[F])_+^{p-2}.
\end{align*}
Substituting this into \eqref{eq:Lp2} and using the mean value theorem one more time we arrive at
\begin{align*}
I(\eta,x)&\leq \Big((f(\eta+\delta_x)-\EE[F])_+-(f(\eta)-\EE[F])_+\Big) \Big((f(\eta+\delta_x)-\EE[F])_+^{p-1}-(f(\eta)-\EE[F])_+^{p-1}\Big)\\
&\leq (p-1)\Big[(f(\eta+\delta_x)-\EE[F])_+-(f(\eta)-\EE[F])_+\Big]^2 c^{p-2}, 
\end{align*}
where $c\in J$. Let us estimate the difference $ \Big[(f(\eta+\delta_x)-\EE[F])_+-(f(\eta)-\EE[F])_+\Big]^2$. We again consider a number of cases.
\begin{enumerate}
    \item If $f(\eta+\delta_x), f(\eta)\ge \EE[F]$, then 
    \[
    \Big[(f(\eta+\delta_x)-\EE[F])_+-(f(\eta)-\EE[F])_+\Big]^2=(D_x F)^2.
    \]
    \item If $f(\eta+\delta_x)\ge \EE[F] \ge f(\eta)$, then 
    \[
    \Big[(f(\eta+\delta_x)-\EE[F])_+-(f(\eta)-\EE[F])_+\Big]^2=\Big[f(\eta+\delta_x)-\EE[F]\Big]^2\leq (D_x F)^2.
    \]
    \item If $f(\eta)\ge \EE[F] \ge f(\eta+\delta_x)$, then 
    \[
    \Big[(f(\eta+\delta_x)-\EE[F])_+-(f(\eta)-\EE[F])_+\Big]^2=\Big[f(\eta)-\EE[F]\Big]^2\leq (D_x F)^2.
    \]
    \item If $\EE[F] \ge f(\eta+\delta_x), f(\eta)$, then 
    \[
    \Big[(f(\eta+\delta_x)-\EE[F])_+-(f(\eta)-\EE[F])_+\Big]^2=0\leq (D_x F)^2.
    \]
\end{enumerate}
Thus
\begin{align*}
I(\eta,x)&\leq (p-1)(D_x F)^2 c^{p-2}\\
&\leq (p-1)(D_x F)^2_{+}(f(\eta+\delta_x)-\EE[F])_+^{p-2}+(p-1)(D_x F)^2_{-}(f(\eta)-\EE[F])_+^{p-2}, 
\end{align*}
and, by plugging this into \eqref{eq:Lp1} and applying Fubini's theorem, we obtain
\begin{align*}
\Ent_{2-2/p}(G)&\leq (p-1)\EE\Big[\int_\XX \Big((D_x F)^2_{+}(f(\eta+\delta_x)-\EE[F])_+^{p-2}+(D_x F)^2_{-}(f(\eta)-\EE[F])_+^{p-2}\Big)\,\mu(\dint x)\Big]\\
&= (p-1)\Big[\int_\XX \EE\Big((D_x F)^2_{+}(f(\eta+\delta_x)-\EE[F])_+^{p-2}\Big)\,\mu(\dint x)\\
&\qquad+\EE\Big((f(\eta)-\EE[F])_+^{p-2}\int_\XX(D_x F)^2_{-}\,\mu(\dint x)\Big)\Big].
\end{align*}
In order to transform the first summand we use the Mecke formula for Poisson processes \cite[Theorem 4.1]{LP} and by the definition of the difference operator we get
\begin{align*}
    &\int_\XX \EE\Big((F(\eta+\delta_x)-F(\eta))^2_{+}(f(\eta+\delta_x)-\EE[F])_+^{p-2}\Big)\,\mu(\dint x)\\
    &\qquad=\EE\int_\XX \Big((F(\eta)-F(\eta-\delta_x))^2_{+}(f(\eta)-\EE[F])_+^{p-2}\Big)\,\eta(\dint x).
\end{align*}
This leads to the inequality
\begin{align*}
\Ent_{2-2/p}(G)&\leq (p-1)\EE\Big[(f(\eta)-\EE[F])_+^{p-2}\Big(\int_\XX \Big((F(\eta)-F(\eta-\delta_x))^2_{+}\,\eta(\dint x)+\int_\XX(D_x F)^2_{-}\,\mu(\dint x)\Big)\Big]\\
&=(p-1)\EE\Big[(F-\EE[F])_+^{p-2}V^+\Big].
\end{align*}
In the last step we apply H\"older's inequality with ${p\over p-2}$ and ${p\over 2}$ to obtain
\begin{align*}
\Ent_{2-2/p}(G)=\|(F-\EE[F])_+\|^p_p-\|(F-\EE[F])_+\|^{p}_{p-1}
\leq(p-1)\|(F-\EE[F])_+\|_p^{p-2}\|V^+\|_{p/2},
\end{align*}
which completes the proof.
\end{proof}

\section{Moment and concentration inequalities for Poisson functionals}\label{sec:MomentConcentrationBounds}

The goal of this section is to derive moment and concentration inequalities for Poisson functionals $F=f(\eta)$, where $\eta$ is a Poisson process with $\sigma$-finite intensity measure $\mu$ over some measurable space $\XX$. In our arguments we essentially follow \cite{BBLM}. In particular, let us introduce the constants
\[
\kappa_p := \frac{1}{2}\Big(1 - \Big(1 - \frac{1}{p}\Big)^{p/2}\Big)^{-1}
\]
for $p > 1$. It is not hard to check that $\kappa_p$ is strictly increasing in $p$ and that
\begin{align*}
    \lim_{p\to 1} \kappa_p = {1\over 2},\qquad\qquad \lim_{p\to \infty} \kappa_p = \frac{\sqrt{e}}{2(\sqrt{e}-1)}=:\kappa.
\end{align*}
%\[
%\kappa := \frac{\sqrt{e}}{2(\sqrt{e}-1)} < 1.271.
%\]
Furthermore, besides $V^+$ defined in Proposition \ref{prop:LpEstimate} we shall need the quantities
\begin{align*}
    V^- &:= \int_\XX (D_x F)_+^2\, \mu (\dint x) + \int_\XX (F(\eta) - F(\eta - \delta_x))_-^2 \,\eta(\dint x),\\
    V &:= \int_\XX (D_x F)^2 \,\mu (\dint x) + \int_\XX (F(\eta) - F(\eta - \delta_x))^2 \,\eta(\dint x).
\end{align*}

\begin{theorem}[Moment bounds]\label{prop:ConcentrationLpBounds}
Let $p\geq 2$ and $F\in L^{1}(\PP_\eta)$. Then
\begin{align}
\lVert (F-\EE[F])_+ \rVert_p &\le \sqrt{2\kappa p \lVert V^+ \rVert_{p/2}} = \sqrt{2\kappa p} \,\lVert \sqrt{V^+} \rVert_{p},\label{momineq+}\\
\lVert (F-\EE[F])_- \rVert_p &\le \sqrt{2\kappa p \lVert V^- \rVert_{p/2}} = \sqrt{2\kappa p} \,\lVert \sqrt{V^-} \rVert_{p}\label{momineq-},\\
\lVert F-\EE[F] \rVert_p &\le \sqrt{8\kappa p \lVert V \rVert_{p/2}} = \sqrt{8\kappa p}\, \lVert \sqrt{V} \rVert_{p}\label{momineq+-}.
\end{align}
\end{theorem}
\begin{proof}
To see \eqref{momineq+}, we prove the slightly sharper estimate
\[
\lVert (F-\EE[F])_+ \rVert_p \le \sqrt{\Big(1-\frac{1}{p}\Big) 2 \kappa_p p \lVert V^+ \rVert_{p/2}}
\]
for any $p \ge 2$.  To this end, setting $c_p := 2 \lVert V^+ \rVert_{{p/2}\vee 1}(1-1/(p\vee2))$, we show by induction on $k$ that for all $k \in \mathbb{N}$ and all $p \in (k, k+1]$,
\begin{equation}\label{Ind}
\lVert (F-\EE[F])_+ \rVert_p \le \sqrt{p \kappa_{p\vee 2} c_p}.
\end{equation}
To start, take $k=1$, $F\in L^1(\PP_\eta)$ and apply H\"older's inequality and the $L^1$-version of the Poincar\'e inequality for Poisson point processes (see \eqref{eq:PoincareIneq} above). Together with the Mecke formula for Poisson processes \cite[Theorem 4.1]{LP}, this yields
\[
\lVert (F-\EE[F])_+ \rVert_p \le \lVert F-\EE[F] \rVert_2 \le \sqrt{\EE\Big[\int_\XX (D_x F)^2 \,\mu(\dint x)\Big]} = \sqrt{\lVert V^+ \rVert_1} \le \sqrt{p \kappa_{p\vee 2} c_p}.
\]

In the induction step, we assume that $F\in L^{p-1}(\PP_\eta)$ and that \eqref{Ind} holds for any integers smaller than some $k > 1$. Writing
\[
x_p := \lVert (F - \EE[F])_+ \rVert_p^p(p\kappa_{p\vee 2}c_p)^{-p/2},
\]
\eqref{Ind} is equivalent to $x_p \le 1$ for any $p \in (k, k+1]$. Moreover, \eqref{Lpest} can be rewritten as
\[
x_pp^{p/2}c_p^{p/2}\kappa_p^{p/2} \le x_{p-1}^{p/(p-1)}(p-1)^{p/2}c_{p-1}^{p/2}\kappa_{(p-1)\vee 2}^{p/2} + \frac{1}{2} x_p^{1-2/p}p^{p/2}c_p^{p/2}\kappa_p^{p/2-1}.
\]
Using $c_{p-1} \le c_p$ and $\kappa_{(p-1)\vee 2} \le \kappa_p$, this reduces to
\[
x_p \le x_{p-1}^{p/(p-1)}\Big(1 - \frac{1}{p}\Big)^{p/2} + \frac{1}{2\kappa_p} x_p^{1-2/p}.
\]
By induction, $x_{p-1} \le 1$, and hence it follows that
\[
x_p \le \Big(1 - \frac{1}{p}\Big)^{p/2} + \frac{1}{2\kappa_p} x_p^{1-2/p}.
\]
Noting that the function
\[
f_p (x):= \Big(1 - \frac{1}{p}\Big)^{p/2} + \frac{1}{2\kappa_p}x^{1-2/p} - x
\]
is strictly concave on $\mathbb{R}_+$, positive at $x=0$, $f_p(1) = 0$ and $f_p(x_p) \ge 0$, we obtain that $x_p \le 1$, which completes the proof of \eqref{momineq+}.

The result in \eqref{momineq-} simply follows by considering $-F$ instead of $F$ in \eqref{momineq+}. Now \eqref{momineq+-} is easily seen by triangle inequality, combining \eqref{momineq+} and \eqref{momineq-} and using that $0 \le V^+, V^- \le V$:
\[
\lVert F-\EE[F] \rVert_p \le \lVert (F-\EE[F])_+ \rVert_p + \lVert (F-\EE[F])_- \rVert_p \le 2 \sqrt{2\kappa p \lVert V \rVert_{p/2}}.
\]
This completes the proof.
\end{proof}

\begin{remark}\rm 
Theorem \ref{prop:ConcentrationLpBounds} yields the same $L^p$-bounds as \cite[Proposition 4.17]{AdamczakEtAl} up to constants, where $\sqrt{2\kappa}$ in \eqref{momineq+} and $\sqrt{8\kappa}$ in \eqref{momineq+-} are both replaced by $D_{4.17} = \sqrt{6\kappa}$.
\end{remark}

In particular, from the previous proposition we conclude the following concentration bounds for Poisson functionals, which recover Proposition 1.2, Corollary 3.3 (ii) and Corollary 3.4 (ii) from \cite{BP} (up to absolute constants).

\begin{corollary}[Concentration bounds]\label{cortb}
Let $F\in L^1(\PP_\eta)$. If $\PP$-almost surely $V^+ \le L$ or $V^- \le L$ for some $L > 0$, we have that
\begin{equation}\label{concineq}
\PP (F - \EE[F] \ge t) \le e^{-ct^2/L}\qquad \text{or}\qquad \PP (F - \EE[F] \le -t) \le e^{-ct^2/L}
\end{equation}
for every $t \ge 4\sqrt{\kappa}$ and $c = \log(2)/(8\kappa)$, respectively. Moreover, if $\PP$-almost surely $V \le L$ for some $L > 0$, we have that
\begin{equation}\label{concineq2}
\PP (|F - \EE[F]| \ge t) \le 2e^{-ct^2/L}
\end{equation}
for every $t \ge 0$ and $c = \log(2)/(16\kappa)$.
\end{corollary}
\begin{proof}
This follows from standard arguments. We include the proof to demonstrate how the constants we obtain come up. For simplicity and without loss of generality, let us assume that $L=1$. First note that by Markov's inequality and Theorem \ref{prop:ConcentrationLpBounds},
\[
\PP (F- \EE[F] \ge t) \le \inf_{p \ge 2} \frac{\EE[(F-\EE[F])_+^p]}{t^p} \le \inf_{p \ge 2} \Big(\frac{\sqrt{2\kappa p}}{t}\Big)^p.
\]
Here, $\sqrt{2\kappa p}/t \le 1/2$ iff $p \le t^2/(8\kappa)$. Therefore, if $t^2/(8\kappa) \ge 2$, plugging in leads to
\[
\PP (F- \EE[F] \ge t) \le 2^{-t^2/(8\kappa)} = \exp\Big(-\frac{\log(2)}{8\kappa} t^2\Big).
\]
By the same arguments we may prove the bound on $\PP (F- \EE[F] \le -t)$ and consequently that
\[
\PP (|F- \EE[F]| \ge t) \le 4\exp\Big(-\frac{\log(2)}{8\kappa} t^2\Big)
\]
for every $t \ge 0$, where we have chosen the factor $4$ (instead of $2$) to extend the bound trivially to the range $t^2/(8\kappa) < 2$. This may be adjusted to the bound stated above by recalling that for any two constants $\gamma_1 > \gamma_2 > 1$, we have that for all $r \ge 0$ and $\gamma > 0$
\[
\gamma_1 \exp(-\gamma r) \le \gamma_2 \exp\Big(-\frac{\log(\gamma_2)}{\log(\gamma_1)}\gamma r\Big)
\]
whenever the left hand side is smaller or equal to $1$.
\end{proof}

\begin{remark}\rm 
As demonstrated in the proof of \eqref{concineq2} it is possible to extend the range of \eqref{concineq} to $t\geq 0$ by adding a prefactor $2$ to the right hand sides of both inequalities and choosing $c = \log(2)/(16\kappa)$. We decided to keep the restricted range of possible values for $t$ here and in what follows.
\end{remark}

\begin{remark}[Boundedness of $F$ vs.\ boundedness of $V^+$]\label{rm:BoundVF}\rm
Let us discuss the assumptions from Corollary \ref{cortb} in more detail. Here we focus on the condition $V^+ \le L$ and consider the simplest possible situation $\XX = \{x\}$, i.e., $\eta$ is a Poisson random variable with mean $\lambda > 0$. Clearly,
\[
V^+ = \lambda (F(\eta + 1) - F(\eta))_-^2 + \eta (F(\eta) - F(\eta -1))_+^2,
\]
and if we moreover assume $F$ to be non-decreasing, it therefore suffices to examine the quantities
\[
V^+(k) = k(F(k)-F(k-1))^2,\qquad k \in \mathbb{N}.
\]
First, the condition $V^+\leq L$ fails for the function $F(k) = k$, which is a priori clear since it would yield sub-Gaussian tails for a Poisson variable. Formally, $V^+(k) = k$, i.e., $V^+$ is unbounded. In this situation, the condition $V^+ \le L$ guarantees a suitable ``deformation'' of the Poisson random variable under consideration which in some sense ``flattens'' its tails. One may wonder whether boundedness of $V^+$ and boundedness of $F$ coincide. It turns out that neither implication is true.

To see this, first consider the function $F(k) := \sum_{j=1}^k \frac{1}{j}$. Clearly, $F$ is unbounded. However, it is easy to check that $V^+(k) = 1/k$ for any $k$, i.e., $V^+$ is bounded. On the other hand, consider the function
\[
F(k) = \sum_{j=1}^{\lfloor k^{1/5}\rfloor} \frac{1}{\lfloor j^{1/5}\rfloor^2}.
\]
By definition, $F(k)$ only changes its value if $k = m^5$ for some natural number $m$ and in this case, the value $1/m^2$ is added. In particular, $F$ is bounded. However, if $k = m^5$, we have $V^+(m^5) = m$, so that $V^+$ is unbounded.
\end{remark}

\begin{remark}[Self-bounding Poisson functionals]\rm 
As demonstrated in \cite[Proposition 4.18]{AdamczakEtAl}, Theorem \ref{prop:ConcentrationLpBounds} directly implies $L^p$-estimates and concentration inequalities for self-bounding Poisson functionals, where the latter also correspond to \cite[Corollary 3.6]{BP}. In fact, if $F\in L^1(\PP_\eta)$ and $\PP$-almost surely $V^+\leq cF^\alpha$ for some $\alpha\in[0,2)$ and $c\in[0,\infty)$, then $\|(F-\EE[F])_+\|_p \leq 2\sqrt{2c\kappa p}\,(\EE[F])^{\alpha\over 2} + (8c\kappa p)^{1\over 2-\alpha}$ for any $p\geq 2$ and
\begin{align*}
    \PP(F\geq \EE[F]+t) \leq 2\exp\Big\{-C\min\Big({t^2\over  c\cdot(\EE[F])^\alpha},{t^{2-\alpha}\over c}\Big)\Big\}
\end{align*}
for every $t\geq 0$, where $C\in(0,\infty)$ is an absolute numerical constant.
\end{remark}

%As demonstrated in \cite[Proposition 4.18]{AdamczakEtAl}, Theorem \ref{prop:ConcentrationLpBounds} directly implies $L^p$-estimates and concentration inequalities for self-bounding Poisson functionals, where the latter also correspond to \cite[Corollary 3.6]{BP}. For simplicity we focus on results for the upper tail only.

%\begin{corollary}\label{cor:SelfBound}
%Let $F\in L^1(\PP_\eta)$ and suppose that $\PP$-almost surely $V^+\leq cF^\alpha$ for some $\alpha\in[0,2)$ and $c\in[0,\infty)$. Then
%\begin{align*}
%    \|(F-\EE[F])_+\|_p \leq 2\sqrt{2c\kappa p}\,(\EE[F])^{\alpha\over 2} + (8c\kappa p)^{1\over 2-\alpha}
%\end{align*}
%for any $p\geq 2$ and
%\begin{align*}
%    \PP(F\geq \EE[F]+t) \leq 2\exp\Big\{-C\min\Big({t^2\over  c\cdot(\EE[F])^\alpha},{t^{2-\alpha}\over c}\Big)\Big\}
%\end{align*}
%for every $t\geq 0$, where $C\in(0,\infty)$ is an absolute numerical constant.
%\end{corollary}

\section{Applications to stochastic geometry models}\label{sec:Appl}

Our goal in this section is to demonstrate how the general results obtained in the previous section can be applied to concrete problems arising in stochastic geometry. We will discuss two models, namely Poisson polytopes and Poisson cylinders and although concentration inequalities for a number of geometric functionals of these two models are already known in the literature, some of the concentration bounds we obtain are new or better than the already existing ones. For both applications we will discuss the relation between our results and those known from the literature. 

\subsection{Poisson polytopes}

\begin{figure}
    \centering
    \includegraphics[width=0.25\columnwidth]{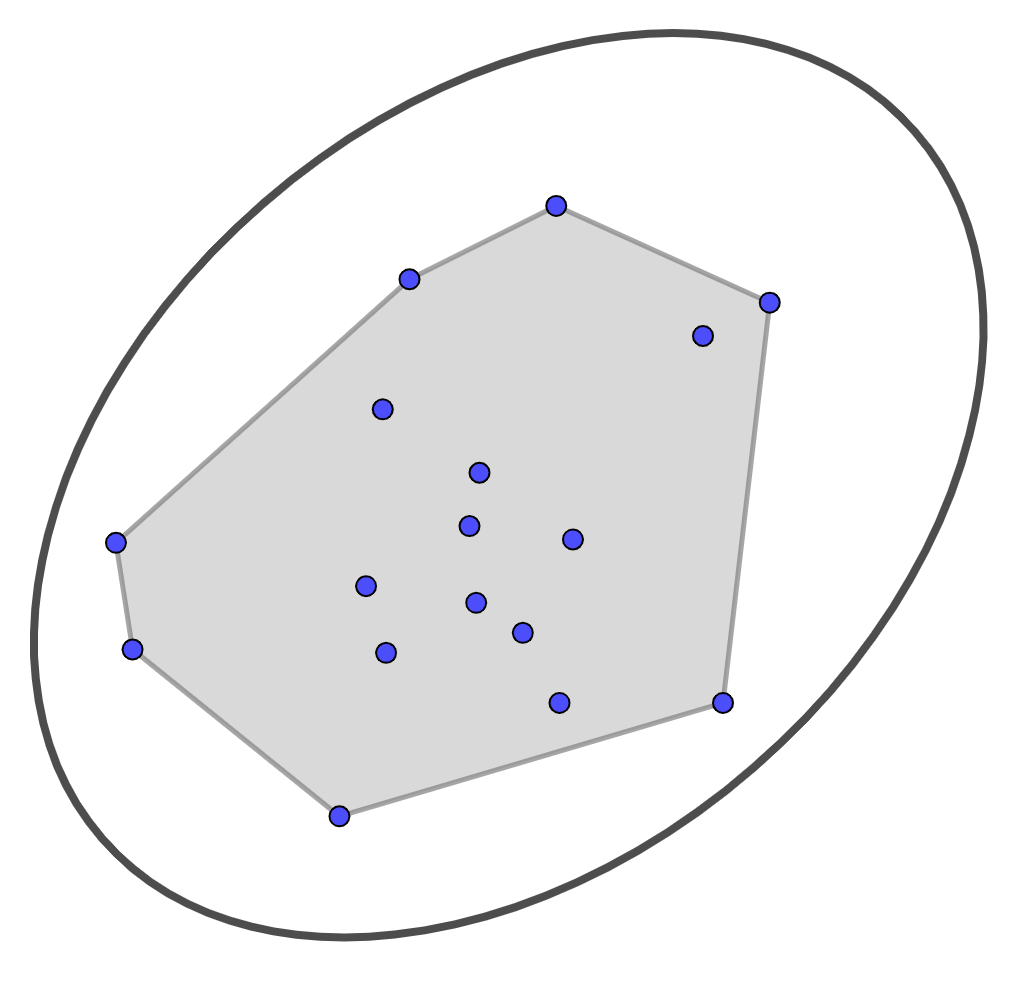}\hspace{3cm}
    \includegraphics[width=0.25\columnwidth]{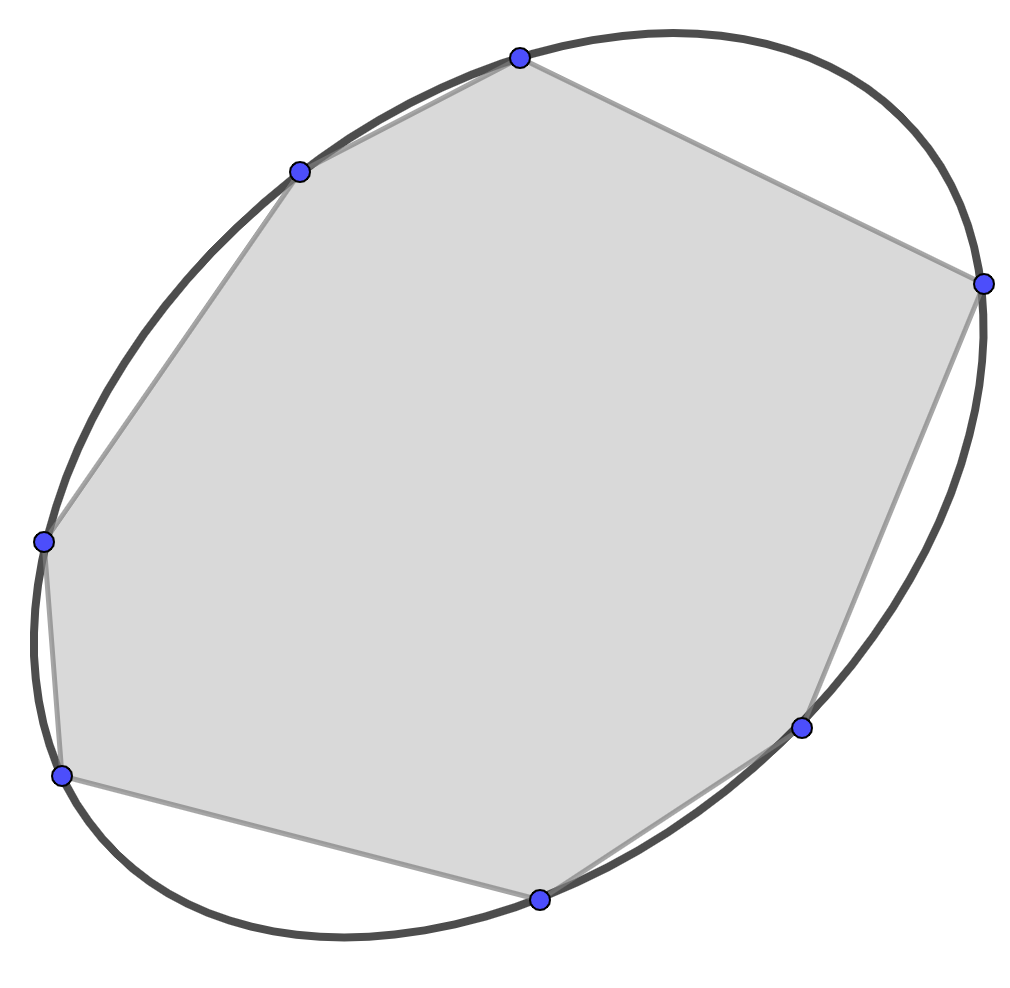}
    \caption{Illustration of the random polytope models (In) (left panel) and (Bd) (right panel) for $d=2$ and with $K$ being an ellipse.}
    \label{fig:Polytopes}
\end{figure}

Our first application is concerned with convex hulls of Poisson processes in $\RR^d$, $d\geq 1$. To describe the set-up, in comparison to the existing literature we take a rather general point of view, and let $\mu$ and $\zeta$ be probability measures on $\RR^d$ with the property that $\mu(H)=0$ and $\zeta(H)=0$ for each hyperplane $H\subset\RR^d$. Now, consider a Poisson process in $\RR^d$ with intensity measure $\gamma\mu$, where $\gamma\in(0,\infty)$ is some fixed intensity parameter. We are interested in concentration properties of the Poisson functional
\begin{equation}\label{eq:FConvHull}
F:=\zeta(\conv(\eta)),    
\end{equation}
where $\conv(\eta)$ stands for the convex hull of the support of $\eta$. Note that under our assumption on $\mu$, $\conv(\eta)$ is $\PP$-almost surely a simplicial polytope in $\RR^d$, meaning that each of its faces is a simplex. In what follows we will often use the notation $\cF_i(P)$, $i\in\{0,1,\ldots,d-1\}$, for the set of $i$-dimensional faces of a polytope $P\subset\RR^d$. Especially $\cF_0(P)$ is the set of vertices, $\cF_1(P)$ the set of edges and $\cF_{d-1}(P)$ the set of facets of $P$.

\begin{proposition}[Concentration for the $\zeta$-content]\label{prop:ConvexHull}
For $F$ as defined by \eqref{eq:FConvHull} one has that
$$
\PP(F\geq \EE[F]+t) \leq e^{-ct^2/(d+1)}\qquad\textup{and}\qquad \PP(F\leq \EE[F]-t)\leq e^{-ct^2/\gamma}
$$
for every $t\geq 4\sqrt{\kappa}$, where $c = \log(2)/(8\kappa)$.
\end{proposition}
\begin{proof}
We start by investigating $V^+$, which in our case is given by
\begin{align*}
    V^+ &= \gamma\int_{\RR^d}(D_xF)_-^2\,\mu(\dint x)+\int_{\RR^d}(F(\eta)-F(\eta-\delta_x))_+^2\,\eta(\dint x)=\int_{\RR^d}(F(\eta)-F(\eta-\delta_x))_+^2\,\eta(\dint x),
\end{align*}
since $\PP$-almost surely $D_xF\geq 0$ for $\mu$-almost all $x\in\RR^d$. It should be mentioned that in case when there are less than $d+1$ points in the support of $\eta$ the convex hull $\conv(\eta)$ is contained in some hyperplane of $\RR^d$ and by our assumptions on the probability measure $\zeta$ we have $F\equiv 0$ $\PP$-almost surely, so that the result becomes trivial. As a consequence, from now on we will assume that there are at least $d+1$ points in the support of $\eta$. For $x\in\eta$ let $N_1(x,\eta)$ be the set of all vertices of $\conv(\eta)$, which are connected with $x$ by an edge, and let $N_2(x,\eta)$ be the set of all points of the process $\eta$, which are contained in the interior of $\conv(x\cup N_1(x,\eta))\setminus \conv(N_1(x,\eta))$. Clearly, $N_1(x,\eta)=\varnothing$ iff $x$ is not on the boundary of $\conv(\eta)$. Also, $N_2(x,\eta)$ can be empty in general. Further we denote by $M(x,\eta)$ the set of facets (i.e.\ $(d-1)$-dimensional faces) of $\conv(\eta\setminus (N_2(x,\eta)\cup \{x\}))$, which are not facets of $\conv(\eta)$. In other words, $M(x,\eta)$ is the set of all facets of the random polytope, which "appear" after removing $x$ from $\eta$ together with the points in $N_2(x,\eta)$. It should be mentioned, that since $\mu(H)=0$ for each hyperplane $H\subset\RR^d$, each facet from $M(x,\eta)$ is $\PP$-almost surely a $(d-1)$-dimensional simplex.

Next, we consider the set of simplices $S(x,\eta):=\{\conv(x,E)\colon E\in M(x,\eta)\}$, that is, simplices with apex $x$ whose bases are facets from $M(x,\eta)$. We now observe that $\PP$-almost surely 
$$
F(\eta)-F(\eta-\delta_x)\leq \zeta\Big(\bigcup_{S\in S(x,\eta)}S\Big)
$$
for any $x\in\eta$. Since the simplices in $S(x,\eta)$ all have disjoint interiors, it follows that
$$
F(\eta)-F(\eta-\delta_x)\leq \sum_{S\in S(x,\eta)}\zeta(S),
$$
and hence
\begin{align*}
    V^+ \leq \int_{\RR^d}  \zeta\Big(\bigcup_{S\in S(x,\eta)}S\Big)^2\,\eta(\dint x) \leq \int_{\RR^d} \sum_{S\in S(x,\eta)}\zeta(S)\,\eta(\dint x)=\sum_{x\in\eta}\sum_{S\in S(x,\eta)}\zeta(S),
\end{align*}
where we used that $\zeta(\bigcup_{S\in S(x,\eta)}S)\leq 1$. Considering the union $S(\eta):=\bigcup_{x\in\eta}S(x,\eta)$ we note, that each simplex $S\in S(\eta)$ is counted at most $d+1$ times in the above sum, according to the number of its vertices. Thus,
$$
V^+ \leq (d+1)\sum_{S\in S(\eta)}\zeta(S).
$$
It is clear, that all vertices of a simplex $S\in S(\eta)$ belong to $\cF_0(\conv(\eta))$. Moreover, by construction each simplex $S\in S(\eta)$ has at least one facet $E$ with corresponding opposite vertex $x$, such that there are no other vertices of the $\conv(\eta)$ lying on the same side of the affine hull generated by $E$, where $x$ is lying. This means, that if the intersection of two simplices $S_1,S_2\in S(\eta)$ has non-empty interior, then they have a common vertex $x$ and, hence, they coincide. Thus, all simplices in $S(\eta)$ have disjoint interiors which together with the assumption that $\zeta(H)=0$ for each hyperplane $H$ leads to the bound $V^+\leq d+1$.
Thus, the result follows from Corollary \ref{cortb} by choosing $L=d+1$ there.

In the same manner we can consider $V^{-}$ and since $\PP$-almost surely $D_xF\geq 0$ for $\mu$-almost all $x\in\RR^d$ we have
$$
V^{-}=\gamma\int_{\RR^d}(D_xF)_+^2\,\mu(\dint x).
$$
Since $\zeta$ is a probability measure on $\RR^d$ we have 
$$
D_xF=\zeta(\conv(\eta+\delta_x))-\zeta(\conv(\eta))\leq 1
$$
and, thus,
$$
V^{-}\leq\gamma.
$$
Applying Corollary \ref{cortb} with $L=\gamma$ we obtain a lower tail.
\end{proof}

Let us consider two specific cases, well known from the existing literature, to which Proposition \ref{prop:ConvexHull} can be applied, see Figure \ref{fig:Polytopes} for illustrations. Here and in the sequel, $\ell_d$ denotes the Lebesgue measure on $\RR^d$.
\begin{itemize}
\item[(In)] Let $K\subset\RR^d$ be a convex body with volume $\ell_d(K)\in(0,\infty)$ and choose for $\mu$ and $\zeta$ the restriction of the $d$-dimensional Lebesgue measure to $K$, normalized by $\ell_d(K)^{-1}$. In this case sub-Gaussian concentration bounds for $F$ are known from \cite{Vu}, but they include a polynomial error term, see also \cite{ReitznerSurvey}; for $K$ being the $d$-dimensional unit ball, exponential concentration inequalities without polynomial error terms are the content of \cite{GroteBall}. We remark that for the random polytope model (In) one has that
$$
\ell_d(K)-\EE[\ell_d(\conv(\eta))]=c_{d,K}\gamma^{-2/(d+1)}(1+o(1)),\qquad\textup{as}\qquad\gamma\to\infty,
$$
if $K$ has a $C^2$-smooth boundary with everywhere strictly positive Gaussian curvature and that 
$$
\ell_d(K)-\EE[\ell_d(\conv(\eta))]=c_{d,K}'\gamma^{-1}(\log\gamma)^{d-1}(1+o(1)),\qquad\textup{as}\qquad\gamma\to\infty,
$$
if $K$ is a polytope. Here, $c_{d,K}$ and $c_{d,K}'$ are explicitly known constants depending on $d$ and on $K$, see \cite{ReitznerSurvey}.
\item[(Bd)] Let $K\subset\RR^d$ be a convex body with volume $\ell_d(K)\in(0,\infty)$ whose boundary is a $C^2$-smooth submanifold with everywhere strictly positive Gaussian curvature and choose for $\mu$ the restriction of the $(d-1)$-dimensional Hausdorff measure $\sH^{d-1}$ to the boundary $\partial K$ of $K$, normalized by $\sH^{d-1}(K)^{-1}$. For $\zeta$ we choose, as for model (In), the normalized Lebesgue measure on $K$. Sub-Gaussian concentration bounds with again polynomial error terms for $F$ in this situation were proved in \cite{RVW}, purely exponential inequalities were not known until now. Again, we remark that in this case 
$$
\ell_d(K)-\EE[\ell_d(\conv(\eta))]=c_{d,K}''\gamma^{-2/(d-1)}(1+o(1)),\qquad\textup{as}\qquad\gamma\to\infty,
$$
where $c_K''$ is another constant depending on $d$ and $K$, see \cite{ReitznerSurvey,RVW}.
\end{itemize}

Suppose now that we are in one of the two situations just described. In this case we can also consider the so-called intrinsic volumes of $\conv(\eta)$. In general, for a compact convex subset $K\subset\RR^d$ the \textbf{$i$-th intrinsic volume} of $K$, $i\in\{0,1,\ldots,d\}$, is defined as
$$
V_i(K) := {d\choose i}{\ell_d(B^d)\over\ell_i(B^i)\ell_{d-i}(B^{d-i})}\EE[\ell_i(K|E)]
$$
where $B^k$, $k\in\NN$, stands for the $k$-dimensional unit ball, $E$ is a uniformly distributed $i$-dimensional random linear subspace of $\RR^d$ and $K|E$ denotes the orthogonal projection of $K$ onto $E$. In particular, $V_d(K)=\ell_d(K)$ is the volume of $K$, $\sH^{d-1}(\partial K)=2V_{d-1}(K)$ its surface content, $V_1(K)$ is a constant multiple of the mean width of $K$ and $V_0(K)=1$ as long as $K\neq\varnothing$. The intrinsic volumes admit a unique additive extension to the class of finite unions of convex sets, which is denoted by the same symbol. For this and for further background material explaining the central role of intrinsic volumes in convex geometry we refer the reader to \cite[Chapter 4]{Schneider} or \cite[Chapter 14]{SW}. Our next result provides new concentration bounds for all intrinsic volumes $V_i(\conv(\eta))$, $i\in\{1,\ldots,d-1\}$ of $\conv(\eta)$ for both random polytope models (In) and (Bd), thereby partially generalizing Proposition \ref{prop:ConvexHull} from which a concentration inequality for $V_d(\conv(\eta))$ follows. We remark that in the special case where $K=B^d$ is the $d$-dimensional unit ball weaker concentration bounds for the intrinsic volumes for the random polytope model (In) are known from \cite{GroteBall}.

\begin{proposition}[Concentration for the intrinsic volumes]
Suppose we are in one of the situations (In) or (Bd) and consider the Poisson functional $F:=V_i(\conv(\eta))$ for some $i\in\{1,\ldots, d-1\}$. Then
$$
\PP(F\geq \EE[F]+t) \leq \exp\Big\{-{c\,t^2\over (i+1)V_i(K)^2}\Big\},\qquad \PP(F\leq \EE[F]-t) \leq \exp\Big\{-{c\,t^2\over \gamma\,V_i(K)^2}\Big\}
$$
for every $t\geq 4\sqrt{\kappa}$, where $c = \log(2)/(8\kappa)$.
\end{proposition}
\begin{proof}
Let $\mu$ be the normalized Lebesgue measure on $K$ if we consider the model (In) and the normalized $(d-1)$-dimensional Hausdorff measure on $\partial K$ in case of the model (Bd). First of all we note that since the intrinsic volumes are monotone under set inclusion on the space of convex sets we have that, $\PP$-almost surely and for $\mu$-almost all $x\in \supp(\mu)$,
\begin{align*}
D_xF=V_i(\conv(\eta+\delta_x))-V_i(\conv(\eta))\ge 0.
\end{align*}
Suppose from now on that $\supp(\eta)$ consists of at least $d+1$ distinct points (the other cases will be discussed below). Next, we recall from \cite[Equation (14.14)]{SW} that for a polytope $P\subset\RR^d$ the $i$-th intrinsic volume can be represented as
$$
V_i(P) = \sum_{G\in\cF_i(P)}\ell_i(G)\,\gamma(G,P),
$$
where we recall that $\cF_i(P)$ is the set of all $i$-dimensional faces of $P$ and $\gamma(G,P)$ stands for the external angle at the face $G$ with respect to $P$. The latter is defined as the solid angle of the convex cone of normal vectors of $P$ at an arbitrary point from the relative interior of $F$, see \cite[Equation (14.10)]{SW}. Applying this formula to the almost surely $d$-dimensional polytopes $P:=\conv(\eta)$ and $P_x:=\conv(\eta-\delta_x)$, where $x\in\supp(\eta)$, we deduce that \begin{align*}
  F(\eta)-F(\eta-\delta_x) &= \sum_{G\in\cF_i(P)}\ell_i(G)\,\gamma(G,P) - \sum_{G\in\cF_i(P_x)}\ell_i(G)\,\gamma(G,P_x)\\
  &= \sum_{G\in\cF_i(P)\atop G\not\in \cF_i(P_x)}\ell_i(G)\,\gamma(G,P) + \sum_{G\in\cF_i(P)\cap\cF_i(P_x)}\ell_i(G)\,[\gamma(G,P)-\gamma(G,P_x)]\\
  &\hspace{4cm}\;- \sum_{G\in\cF_i(P_x)\atop G\not\in \cF_i(P)}\ell_i(G)\,\gamma(G,P_x)\\
  &\leq \sum_{G\in\cF_i(P)\atop G\not\in \cF_i(P_x)}\ell_i(G)\,\gamma(G,P) + \sum_{G\in\cF_i(P)\cap\cF_i(P_x)}\ell_i(G)\,[\gamma(G,P)-\gamma(G,P_x)].
\end{align*}
We shall argue now that the second sum in the last expression is non-positive. To do this, we denote by $n_1(G),\ldots,n_{m}(G)\in\RR^d$ the outer unit normal vectors of the facets $E_1,\ldots,E_{m}\in\cF_{d-1}(P)$ of $P$ containing a face $G\in \cF_i(P)\cap \cF_i(P_x)$ and by $\hat n_1(G),\ldots,\hat n_{k}(G)$ the outer unit normal vectors of the facets $\hat E_1,\ldots, \hat E_{k}\in\cF_{d-1}(P_x)$ of $P_x$ containing $G$. Further we note that each of the facets $\hat E_j$, $1\leq j\leq k$, either coincides with one of the facets $E_j$, $1\leq j\leq k$, or its relative interior belongs to the interior of $P$. This in particular means that ${\rm pos}(n_1(G),\ldots,n_{m}(G))\subseteq {\rm pos}(\hat n_1(G),\ldots,\hat n_{k}(G))$, where ${\rm pos}(\,\cdot\,)$ stands for the positive hull. Thus,
$$
\gamma(G,P) = \alpha_{d-i-1}({\rm pos}(\hat n_1(G),\ldots,\hat n_{m}(G)))\leq \alpha_{d-i}({\rm pos}(\hat n_1(G),\ldots,\hat n_{k}(G))=\gamma(G,P_x),
$$
with $\alpha_{k}(\,\cdot\,):=\cH^{k-1}(\,\cdot\,\cap\SS^{d-1})/\omega_{k}$ and with $\omega_k:=\sH^k(\SS^k)$ denoting the $k$-dimensional Hausdorff measure of the $k$-dimensional unit sphere $\SS^k$. This shows that, indeed, the second sum in the above bound for $F(\eta)-F(\eta-\delta_x)$ is non-positive, and hence that
$$
F(\eta)-F(\eta-\delta_x) \leq \sum_{G\in\cF_i(P)\atop x\in G}\ell_i(G)\,\gamma(G,P)
$$
$\PP$-almost surely for $x\in\supp(\eta)$. Then by monotonicity of the intrinsic volumes under set inclusion on the space of convex sets we conclude that $\PP$-almost surely,
\begin{align*}
V^+ &= \int_{\RR^d} (F(\eta)-F(\eta-\delta_x))_+^2\,\eta(\dint x)\\
&\leq V_i(P)\int_{\RR^d} \sum_{G\in\cF_i(P)\atop x\in G}\ell_i(G)\,\gamma(G,P)\,\eta(\dint x)\\
&\leq V_i(K)\int_{\RR^d} \sum_{G\in\cF_i(P)\atop x\in G}\ell_i(G)\,\gamma(G,P)\,\eta(\dint x)\\
&= V_i(K)\int_{\RR^d} \sum_{G\in\cF_i(\conv(\eta))\atop x\in G}\ell_i(G)\,\gamma(G,\conv(\eta))\,\eta(\dint x),
\end{align*}
where in the first step we used the fact that external angles are non-negative. By construction of the models (In) and (Bd), each face $G\in\cF_i(\conv(\eta))$ is $\PP$-almost surely an $i$-dimensional simplex with precisely $i+1$ vertices. Thus,
\begin{equation}\label{eq:270421}
V^+ \leq (i+1)V_i(K)\sum_{G\in\cF_i(\conv(\eta))}\ell_i(G)\,\gamma(G,\conv(\eta)) = (i+1)V_i(K)V_i(\conv(\eta))\leq (i+1)V_i(K)^2,
\end{equation}
once again by the monotonicity of the intrinsic volumes under set inclusion on the space of convex subsets of $\RR^d$.

We shall now remove the assumption that $\supp(\eta)$ consists of at least $d+1$ points. First, if $\supp(\eta)$ has less than $i+1$ points the result is trivial since $F=V_i(\conv(\eta))$ is zero in this case. On the other hand, if $\supp(\eta)$ has more than $i+1$ points, a straight forward adaption of the above argument still works, basically the only change is that the polytopes $P$ and $P_x$ have dimension $i$ instead of $d$ in this case. Especially the bound \eqref{eq:270421} does not change. Finally, if $\supp(\eta)$ consists of precisely $i+1$ points, $\conv(\eta)$ is an $i$-dimensional simplex, in which case $V_i(\conv(\eta))-V_i(\conv(\eta-\delta_x))=V_i(\conv(\eta))$ $\PP$-almost surely for each of the $i+1$ vertices $x\in\eta$. By definition of $V^+$ this implies in the situation under consideration that $\PP$-almost surely $V^+=(i+1)V_i(\conv(\eta))^2\leq (i+1) V_i(K)^2$. So, \eqref{eq:270421} still holds in this set-up. Putting together all these cases, the bound for the upper tail now follows from Corollary \ref{cortb} by choosing $L=(i+1)V_i(K)^2$ there. 

The proof of the lower tail follows from the fact that $\PP$-almost surely
$$
V^-=\int_{\RR^d}(D_xF)^2\,\mu(\dint x)\leq \gamma V_i(K)^2
$$
and once again from Corollary \ref{cortb}, this time applied with $L=\gamma V_i(K)^2$.
\end{proof}

\subsection{Poisson cylinder models}

\begin{figure}
    \centering
    \includegraphics[width=0.45\columnwidth]{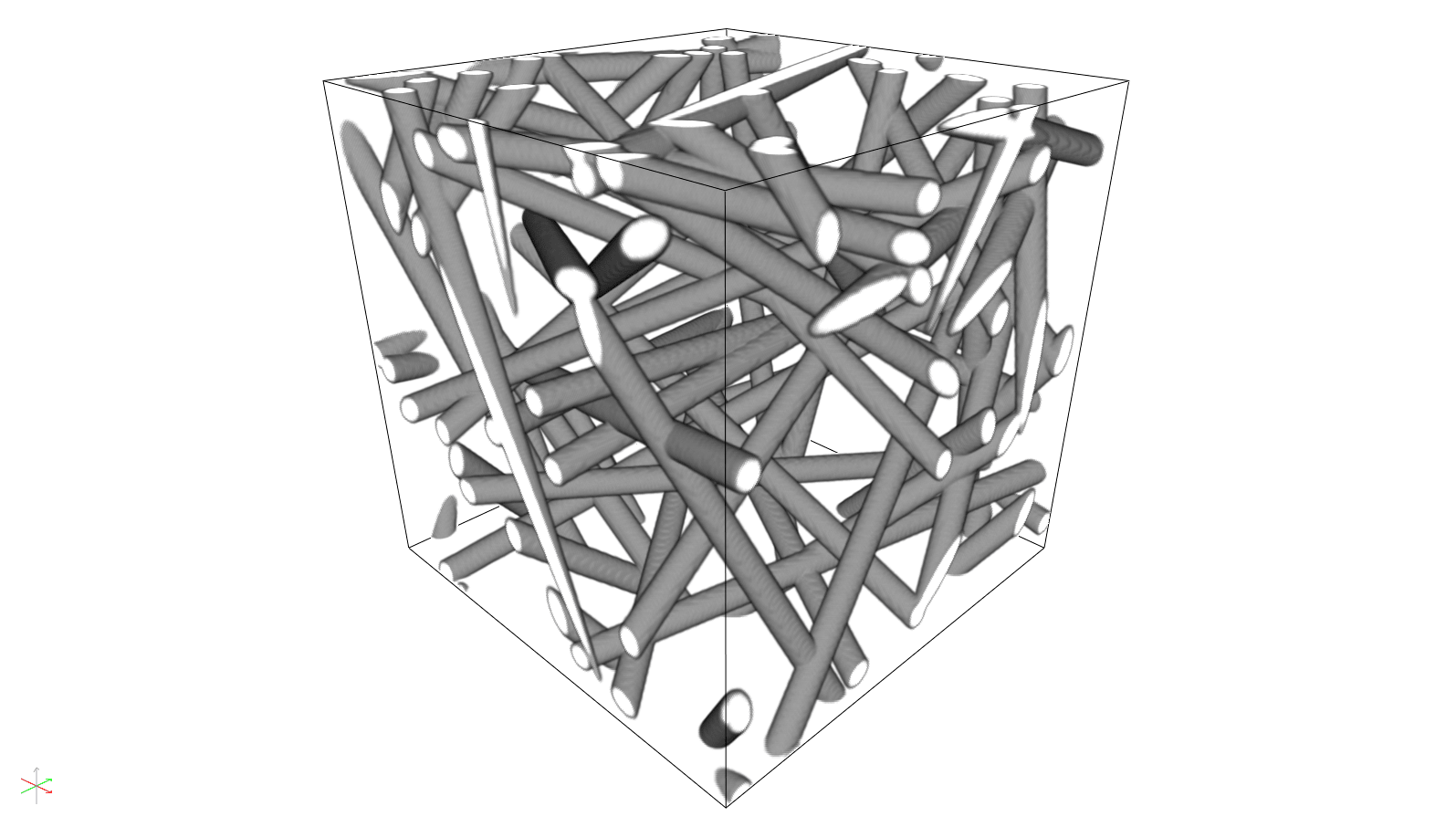}
    \includegraphics[width=0.45\columnwidth]{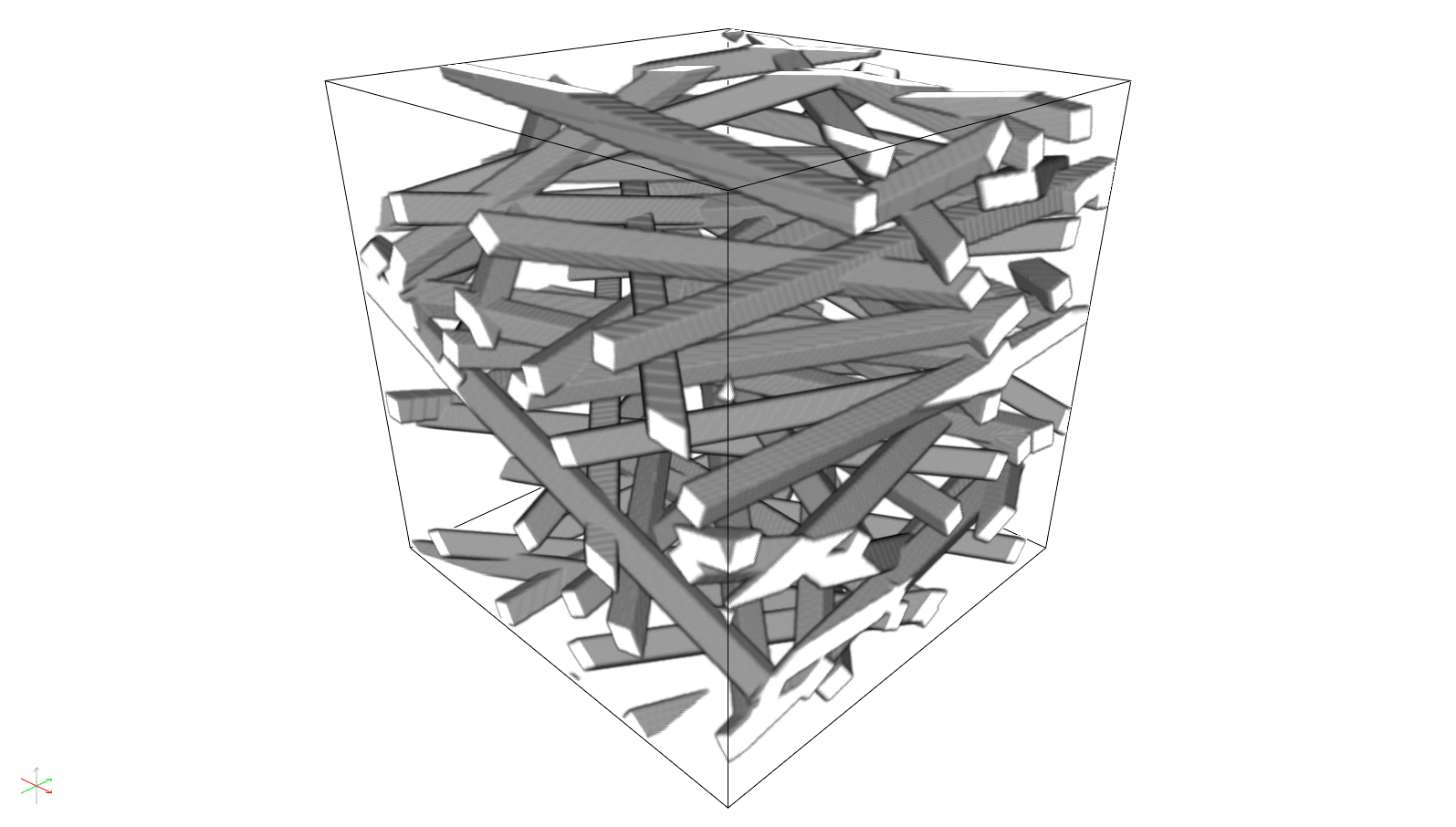}
    \caption{Simulations of two Poisson cylinder models with $d=3$, $k=1$ and with different base distributions (circles in the left panel and squares in the right panel). They were kindly provided by Claudia Redenbach (Kaiserslautern).}
    \label{fig:Cylinders}
\end{figure}

Our second application deals with the Poisson cylinder model for which a number of concentration properties for various geometric functionals have recently been studied in \cite{BaciBetkenGusakovaThaele}. We refer to that paper and the literature cited therein for further references and background material. To define the model, fix some space dimension $d\geq 1$ and another dimension parameter $k\in\{0,1,\ldots,d-1\}$. By $e_1,\ldots,e_d$ we denote the standard orthonormal basis in $\RR^d$ and let $G(d,k)$ be the Grassmannian of all $k$-dimensional linear subspaces of $\RR^d$. We identify each element $L\in G(d,k)$ with a representative of the equivalence class ${\bf M}_L$ of orthogonal matrices $M_L\in{\rm SO}(d)$ satisfying $L=M_LE_k$, where $E_k:={\rm span}(e_{d-k+1},\ldots,e_d)$; for concreteness we choose the lexicographically smallest element ${\rm lex\,min\,}{\bf M}_L$ from ${\bf M}_L$. Then we put $\SO_{d,k}:=\{{\rm lex\,min\,}{\bf M}_L:L\in G(d,k)\}$ and define the product space $\MM_{d,k}:=\SO_{d,k}\times\cC_{d-k}'$, where $\cC_{d-k}'$ denotes the space of non-empty compact subsets of $\RR
^{d-k}={\rm span}(e_1,\ldots,e_{d-k})\subset\RR^d$. The elements of that space describe the direction ($\SO_{d,k}$-component) and the basis ($\cC_{d-k}'$-component) of a cylinder.

Next, we let $\QQ$ be a probability measure on $\MM_{d,k}$ and consider a Poisson process $\eta$ in the product space $\RR^{d-k}\times\MM_{d,k}$ with intensity measure $\gamma\ell_{d-k}\otimes\QQ$, where $\gamma\in(0,\infty)$ is an intensity parameter and $\ell_{d-k}$ denotes the Lebesgue measure on $\RR^{d-k}$. The associated random union set
$$
Z := \bigcup_{(x,\theta,K)\in\eta}Z(x,\theta,K)\qquad\textup{with}\qquad Z(x,\theta,K):=\theta((x+K)\times E_k),
$$
is called a {\bf Poisson cylinder model}, see Figure \ref{fig:Cylinders} for two simulations. The measure $\QQ$ describes the joint distribution of the direction and the basis of a typical cylinder of the model. We emphasize that in the special case where $k=0$ the Poisson cylinder model reduces to the classical {\bf Boolean model}, which is included in our discussion as well, see \cite{SW} for background material on Boolean models. For a test set $W\in\cC_d'$ we are interested in the volume of $Z$ that can be observed in $W$, that is,
\begin{equation}\label{eq:CylinderDef}
F:=\ell_d(Z\cap W).    
\end{equation}
From \cite[Equation (3.1)]{BaciBetkenGusakovaThaele} we have that $\EE[F]=\ell_d(W)(1-e^{-\gamma\EE[\ell_{d-k}(\Xi)]})$. Here, $\Xi$ denotes the so-called typical cylinder base, i.e., a random element with distribution $\QQ_*$, where $\QQ_*$ is the marginal distribution of $\QQ$ on the $\cC_{d-k}'$-coordinate. Concentration properties for $F$ around its mean have recently been studied in \cite[Section 3]{BaciBetkenGusakovaThaele} (and \cite{GieringerLast} for $k=0$) based on concentration inequalities deduced from general covariance identities for Poisson functionals. The purpose of this section is to demonstrate that Corollary \ref{cortb} can also be used to derive concentration bounds for $F$, where in addition to \cite{BaciBetkenGusakovaThaele} we assume that the typical cylinder base has uniformly bounded volume. 

\begin{proposition}[Concentration for the volume]\label{prop:CylinderVol}
Consider a Poisson cylinder model as described above and suppose that the typical cylinder base $\Xi$ satisfies $\ell_{d-k}(\Xi)\leq A$ $\PP$-almost surely for some $A<\infty$. Then for $F$ as defined by \eqref{eq:CylinderDef} one has that
\begin{align*}
   \PP(F\geq\EE[F]+t) &\leq \exp\Big\{-{c\,t^2\over  A\,{\rm diam}(W)^k\,\ell_d(W)}\Big\},\\
   \PP(F\leq\EE[F]-t) &\leq \exp\Big\{-{c\,t^2\over \gamma\,A^2\,{\rm diam}(W)^k\,\ell_d(W)}\Big\}
\end{align*}
for every $t\geq 4\sqrt{\kappa}$, where $c = \log(2)/(8\kappa)$.
\end{proposition}
\begin{proof}
In view of Corollary \ref{cortb} the bound for the upper tail follows once we prove that $V^+\leq A\,{\rm diam}(W)^k\,\ell_d(W)$. To establish this inequality we recall the definition of $V^+$:
\begin{align*}
    V^+ &= \gamma\int_{\MM_{d,k}}\int_{\RR^{d-k}} (D_{(x,\theta,K)}F)_-^2\,\ell_{d-k}(\dint x)\QQ(\dint(\theta,K))\\
    &\qquad+ \int_{\MM_{d,k}}\int_{\RR^{d-k}}(F(\eta)-F(\eta-\delta_{(x,\theta,K)}))_+^2\,\eta(\dint(x,\theta,K)).
\end{align*}
Now, since for $\ell_{d-k}\otimes\QQ$-almost all $(x,\theta,K)\in\RR^{d-k}\times\MM_{d,k}$,
\begin{align*}
    D_{(x,\theta,K)}F &= \ell_d((Z\cup Z(x,\theta,W))\cap W) - \ell_d(Z\cap W)\\
    &=\ell_d(Z(x,\theta,K)\cap W) - \ell_d(Z\cap Z(x,\theta,K)\cap W) \geq 0
\end{align*}
holds $\PP$-almost surely, we have that $(D_{(x,\theta,K)}F)_-^2=0$, implying that the first term in the representation of $V^+$ vanishes. On the other hand, for $(x,\theta,K)\in\eta$ we have that $\PP$-almost surely
\begin{align*}
    F(\eta)-F(\eta-\delta_{(x,\theta,K)}) &= \ell_d((Z_{-(x,\theta,K)}\cup Z(x,\theta,K))\cap W) - \ell_d(Z_{-(x,\theta,K)}\cap W)\\
    &=\ell_d(Z(x,\theta,K)\cap W) - \ell_d(Z_{-(x,\theta,K)}\cap Z(x,\theta,K)\cap W)  \\
    &= \ell_d((Z(x,\theta,K)\cap W)\setminus Z_{-(x,\theta,K)})\geq 0,
\end{align*}
where $Z_{-(x,\theta,K)}:=\bigcup_{(y,\phi,L)\in\eta-\delta_{(x,\theta,K)}}Z(y,\phi,L)$ stands for the cylinder model based on $\eta$ but with the cylinder corresponding to $(x,\theta,K)$ removed. Thus, using the assumption on the boundedness of the volume of the typical cylinder base, we find that $\PP$-almost surely
\begin{align*}
    &\int_{\MM_{d,k}}\int_{\RR^{d-k}}(F(\eta)-F(\eta-\delta_{(x,\theta,K)}))_+^2\,\eta(\dint(x,\theta,K))\\
    &\qquad\leq \int_{\MM_{d,k}}\int_{\RR^{d-k}}\ell_d((Z(x,\theta,K)\cap W)\setminus Z_{-(x,\theta,K)})^2\,\eta(\dint(x,\theta,K))\\
    &\qquad\leq A\,{\rm diam}(W)^k \int_{\MM_{d,k}}\int_{\RR^{d-k}}\ell_d((Z(x,\theta,K)\cap W)\setminus Z_{-(x,\theta,K)})\,\eta(\dint(x,\theta,K)) \\
    %&\qquad = A\,{\rm diam}(W)^k\,F \\
    &\qquad \leq A\,{\rm diam}(W)^k\,\ell_d(W).
\end{align*}
Now, Corollary \ref{cortb} can be applied with $L=A\,{\rm diam}(W)^k\,\ell_d(W)$. This concludes the proof for the upper tail.

To obtain the bound for the lower tail we notice that, since  $F(\eta)-F(\eta-\delta_{(x,\theta,K)})\geq 0$ $\PP$-almost surely for $(x,\theta,K)\in\eta$,
\begin{align*}
    V^- &= \gamma\int_{\MM_{d,k}}\int_{\RR^{d-k}}\big(\ell_d(Z(x,\theta,K)\cap W)-\ell_d(Z\cap Z(x,\theta,K)\cap W)\big)^2\,\ell_{d-k}(\dint x)\QQ(\dint(\theta,K))\\
    &\leq \gamma\int_{\MM_{d,k}}\int_{\RR^{d-k}}\ell_d(Z(x,\theta,K)\cap W)^2\,\ell_{d-k}(\dint x)\QQ(\dint(\theta,K))\\
    &\leq \gamma\,A\,{\rm diam}(W)^k\int_{\MM_{d,k}}\int_{\RR^{d-k}}\ell_d(Z(x,\theta,K)\cap W)\,\ell_{d-k}(\dint x)\QQ(\dint(\theta,K))
\end{align*}
holds $\PP$-almost surely. For fixed $(\theta,K)\in\MM_{d,k}$ the inner integral can be evaluated by Fubini's theorem, which yields
\begin{align*}
    \int_{\RR^{d-k}}\ell_d(Z(x,\theta,K)\cap W)\,\ell_{d-k}(\dint x) = \ell_{d-k}(K)\,\ell_d(W),
\end{align*}
(this formula is also a special case of \cite[Theorem 2]{SW86}, which is stated there under the (in this situation unnecessary) assumption that both $W$ and $K$ are convex). Thus,
$$
V^- \leq \gamma\,A\,{\rm diam}(W)^k\,\ell_d(W)\,\EE[\ell_{d-k}(\Xi)] \leq \gamma\,A^2\,{\rm diam}(W)^k\,\ell_d(W),
$$
where we used our assumption on the volume of the typical cylinder base.
The bound for the lower tail now follows from Corollary \ref{cortb} with $L=\gamma\,A^2\,{\rm diam}(W)^k\,\ell_d(W)$. This completes the proof.
\end{proof}

To compare the result of Proposition \ref{prop:CylinderVol} with \cite[Corollary 4.5]{BaciBetkenGusakovaThaele} we focus on the upper tail and choose a $(d-k)$-dimensional ball $B_{\varrho}^{d-k}$ of radius $\varrho>0$ as our typical cylinder base and for the window $W$ a $d$-dimensional ball of radius $R>0$. In this case \cite[Corollary 4.5]{BaciBetkenGusakovaThaele} yields that
$$
\PP(F\geq \EE[F]+t) \leq \exp\Big\{{t\over a}-\Big(b+{t\over a}\Big)\log\Big(1+{t\over ab}\Big)\Big\}
$$
for all $t\geq 0$, where $a=A\,{\rm diam}(W)^k$ with $A=\ell_{d-k}(B_{\varrho}^{d-k})$ and $b>0$ is another constant depending on $\varrho,R$ and the intensity $\gamma$ whose value is not relevant for our purpose. At first sight, this bound seems worse than the one in Proposition \ref{prop:CylinderVol}. However, since $ab\geq\EE[F]$ as shown in the proof of \cite[Corollary 4.3]{BaciBetkenGusakovaThaele}, we have that
$$
\PP(F\geq \EE[F]+t) \leq \exp\Big\{{t\over a}-{\EE[F]\over a}\Big(1+{t\over\EE[F]}\Big)\log\Big(1+{t\over\EE[F]}\Big)\Big\},\qquad t\geq 0.
$$
Using now the elementary inequality $(1+x)\log(1+x)\geq x+{1\over 2}x^2/(1+x/3)$, valid for $x\geq 0$, we conclude that
\begin{equation}\label{eq:Cyl1Comp}
\PP(F\geq \EE[F]+t) \leq \exp\Big\{-{t^2\over 2a(\EE[F]+{t\over 3})}\Big\},\qquad t\geq 0.    
\end{equation}
To compare \eqref{eq:Cyl1Comp} with the bound from Proposition \ref{prop:CylinderVol} we put $p:=\EE[F]/\ell_d(W)$ and observe that in the relevant regime where $t\leq\ell_d(W)-\EE[F]$ (for larger $t$ the probability $\PP(F\geq\EE[F]+t)$ is zero by construction), one has that $2a(\EE[F]+{t\over 3})\leq 2a(p\ell_d(W)+{1\over 3}(1-p)\ell_d(W))={2\over 3}a\ell_d(W)(2p+1)$. Moreover, the inequality ${2\over 3}a\ell_d(W)(2p+1)\leq{a\over c}\ell_d(W)$ is always satisfied, since ${2\over 3}(2p+1)\leq 2< {1\over c}\approx 14.66$ and $0\leq p\leq 1$. As a consequence, 
$$
\exp\Big\{-{t^2\over 2a(\EE[F]+{t\over 3})}\Big\} \leq \exp\Big\{-{c\,t^2\over a\,\ell_d(W)}\Big\}
$$
for all relevant values of $t$. In other words, the concentration bound from \cite{BaciBetkenGusakovaThaele} is in this case always better than the one implied by Proposition \ref{prop:CylinderVol} by at least a constant factor.

\subsection*{Acknowledgement} 
We would like to thank two anonymous referees for their insightful comments and remarks, which helped us to improve our paper.\\
A.G.\ was partially supported by the Deutsche Forschungsgemeinschaft (DFG) via RTG 2131 \textit{High-Dimensional Phenomena in Probability -- Fluctuations and Discontinuity}. H.S.\ was supported by the Deutsche Forschungsgemeinschaft (DFG) via CRC 1283 \textit{Taming Uncertainty and Profiting from Randomness and Low Regularity in Analysis, Stochastics and their Applications}. C.T.\ was supported by the Deutsche Forschungsgemeinschaft (DFG) via SPP 2256 \textit{Random Geometric Systems}.

\end{document}